\documentclass[a4paper, 12pt]{article}
\usepackage{amsmath,amssymb,amsthm}
\usepackage{hyperref}
\usepackage{graphicx}
\usepackage[a4paper,hmargin=2.5cm,vmargin=2.5cm]{geometry}

\pdfoutput=1

\theoremstyle{definition}
\newtheorem{proposition}{Proposition}
\newtheorem{lemma}[proposition]{Lemma}
\newtheorem*{remark}{Remark}
\newtheorem{theorem}[proposition]{Theorem}

\author{Ivan Frolov}
\date{\today}
\title{On triangulations with fixed areas}

\begin{document}

\maketitle

\begin{abstract}
We prove that the number of dissections of a given polygon into triangles with fixed areas of faces is finite and that an equidissection is algebraic as long as the vertices of the original polygon have algebraic coordinates.
\end{abstract}

\section{Introduction}

A \emph{dissection} of a polygon $P$ in $\mathbb R^2$ is a partition into triangles. A dissection is called an \emph{equidissection} if all triangles have equal areas. By a \emph{triangulation} we mean a dissection which has no vertices on the boundary of $P$ except for the vertices of $P$ and such that the intersection of any two triangles is either a common edge, a common vertex, or the empty set.

In \cite{Mon} P. Monsky proved that every equidissection of a square has an even number of triangles. The ingenious argument uses 2-adic valuation and Sperner's lemma.
Various generalisations have been found, for example for centrally symmetric polygons (\cite{Mon90}).

Motivated by a question asked by R. Kenyon in \cite{Ken} we prove the following.

\begin{theorem}
The number of triangulations of a polygon $P$ in $\mathbb R^2$ with fixed number of faces and fixed areas of faces is finite.
\end{theorem}

Our method also allows applies to a question of Kasimatis and Stein (\cite{KS}, p.293, question~5), giving the following.

\begin{theorem}
Suppose that all vertices of a polygon $P$ have algebraic coordinates. If a triangulation of $P$ is an equidissection, then the coordinates of all vertices of this triangulation are algebraic.
\end{theorem}

\begin{remark}
We believe that in Theorems 1 and 2 it is unnecessary to assume the dissection to be triangulation. However, our argument has to be modified to include more technical details.
\end{remark}

The idea of our proof is to use a projectivisation of the variety of all triangulations with given areas. Similar techniques were used in \cite{MR3148653}, \cite{MR4604929} and \cite{MR4393087} to find nontrivial polynomial relations between areas of the triangles in a square dissection.

\section{Proofs}

Let $(x_i,y_i)$ be the coordinates of a point $v_i$. The oriented area of a triangle $v_1v_2v_3$ satisfies
\[
S_{v_1v_2v_3}=\frac12
\begin{vmatrix}
x_1 & x_2 & x_3 \\
y_1 & y_2 & y_3 \\
1 & 1 & 1
\end{vmatrix}
\]
This formula allows to define the oriented area of a triangle $v_1v_2v_3$, where $v_i \in \mathbb C^2$. For $n$ points $v_1,\ldots, v_n \in \mathbb C^2$ we can define the oriented area of the 'polygon' $v_1\ldots v_n$ as
\begin{equation}
S_{v_1\ldots v_n} = \sum_{i=1}^n S_{v_iv_{i+1}p}
\end{equation}
where $v_{n+1}=v_1$ and $p \in \mathbb C^2$ is arbitrary. The result does not depend on $p$ and coincides with the usual area if $v_1\ldots v_n$ is a polygon in $\mathbb R^2$ oriented anticlockwise.

Consider a polygon $v_1 \ldots v_n$ in $\mathbb R^2$ and a triangulation of it with vertices $v_1, \ldots, v_n$, $ v_{n+1},\ldots,v_N$. We will fix the combinatorial type of the triangulation, i.e.\! the planar graph~$G$ on vertices $v_1,\ldots,v_N$.

We consider the standard embedding $\mathbb R^2 \to \mathbb {CP}^2 \colon (x,y) \to [x:y:1]$ and define the configuration space of triangulations with given areas as the algebraic subvariety $X \subset (\mathbb {CP}^2)^{N-n}$ cut by equations
\begin{equation}
2S_{v_iv_jv_k}z_iz_jz_k=
\begin{vmatrix}
x_i & x_j & x_k \\
y_i & y_j & y_k \\
z_i & z_j & z_k \\
\end{vmatrix}
\end{equation}
for each triangle $v_iv_jv_k$ of $G$, where $[x_i:y_i:z_i]$ are coordinates on $(\mathbb {CP}^2)^{N-n}$ for $i > n$ and constants for $i \le n$, and $S_{v_iv_jv_k}$ is a constant (the area of the triangle $v_iv_jv_k$).

\begin{proposition}
The variety $X$ consists of finitely many points.
\end{proposition}

\begin{proof}
For each $n < i \le N$ consider the projection $\pi_i \colon X \to \mathbb {CP}^2$ sending a point in the configuration space to $[x_i:y_i:z_i]$. It suffices to prove that $\pi_i(X)$ is finite for all $i$. Assume the contrary, i.e.\! there exist $n < m \le N$ such that $\pi_m(X)$ is not finite.

\begin{lemma}
Under the assumptions above $\pi_m(X)$ intersects the line at infinity.
\end{lemma}

\begin{proof}
As an image of a projective variety is always Zariski closed (\cite{Har}, Theorem~3.12), $\pi_m(X)$ is a closed algebraic subvariety of $\mathbb {CP}^2$. Since it has infinitely many points, its dimension is at least~1. So $\pi_m(X)$ intersects the line at infinity (\cite{Har}, Proposition 11.4).
\end{proof}

By Lemma 4 we may choose a collection of points $u_{n+1},\ldots,u_N \in \mathbb {CP}^2$ satisfying (2) such that $u_m$ is at infinity, i.e.\! satisfies $z_m=0$.
For convenience we denote $u_i=v_i$ for $1 \le i \le n$.
After renumbering the indices we may assume that the points $u_1,\ldots,u_{m-1}$ lie in $\mathbb C^2$ and the points $u_m,\ldots,u_N$ lie at infinity. Let $H'$ be the subgraph of $G$ consisting of vertices $v_1,\ldots,v_{m-1}$ and all edges between them. Let $H$ be the connected component of $H'$ containing $v_1, \ldots v_n$.

Let $f$ be a face of $H$. Since $H$ is connected, the boundary of $f$ is connected. Let $v_{i_1} \ldots v_{i_t}$ be the closed path around the boundary of $f$ anticlockwise (one vertex may appear several times, one edge may appear twice in this path).

\begin{lemma}
If $f$ is not a face of $G$, then the points $u_{i_1}, \ldots, u_{i_t}$ are collinear.
\end{lemma}

\begin{proof}
Note that the equation (2) implies that if $v_iv_jv_k$ is a face of $G$ and at least one of the points $u_i$, $u_j$, $u_k$ is at infinity, then $u_i$, $u_j$, $u_k$ are collinear.

\begin{figure}[ht]
\begin{center}
\includegraphics{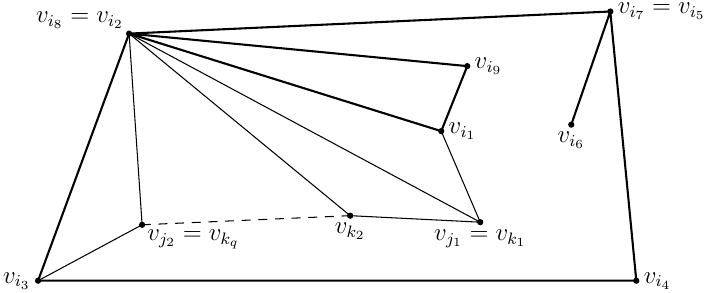}
\end{center}
\end{figure}

For each $1 \le s < t$ let $v_{i_s}v_{i_{s+1}}v_{j_s}$ be the face of $G$ contained in $f$.
We claim that $u_{j_1}=u_{j_2}$. Indeed, let the faces of $G$ covering the angle $v_{i_1}v_{i_2}v_{i_3}$ of $f$ be $v_{i_1}v_{i_2}v_{k_1}$, $v_{k_1}v_{i_2}v_{k_2},\ldots$, $v_{k_q}v_{i_2}v_{i_3}$ where $k_1=j_1$ and $k_q=j_2$. Since $f$ is not a face of $G$, all points $v_{k_\ell}$ lie in the interior of $f$. So all points $u_{k_\ell}$ lie at infinity, but $u_{i_2}$ does not.
Since $v_{k_\ell}v_{i_2}v_{k_{\ell+1}}$ is a face of $G$, the points $u_{k_\ell}$, $u_{i_2}$, $u_{k_{\ell+1}}$ are collinear, which is only possible if $u_{k_\ell}=u_{k_{\ell+1}}$. Thus $u_{j_1}=u_{k_1}=u_{k_2}=\ldots=u_{k_q}=u_{j_2}$.

The same argument shows that $u_{j_{s-1}}=u_{j_s}$ for all $s$. Therefore all $u_{j_s}$ are the same point $u$ at infinity. By induction on $s$ it follows that $u_{i_s}$ lies on the line $u_{i_1}u$ for every~$s$.
\end{proof}

Denote the area of $f$ by $S_f$, it is equal to $S_{v_{i_1} \ldots v_{i_t}}$ interpreted according to (1). Denote $S'_f=S_{u_{i_1} \ldots u_{i_t}}$. Lemma 5 implies that $S'_f=0$ if $f$ is not a face of $G$. If $f$ is a face of $G$, then (2) implies that $S'_f=S_f$.

We have $\sum_f S_f = S_{v_1\ldots v_n}$, sum taken over all bounded faces of $H$ oriented anticlockwise.
Since it is a polynomial equation on the coordinates of the points $v_i$, it is true when $v_i$ are replaced by arbitrary points of $\mathbb C^2$. For every vertex $v_i$ of $H$ the point $u_i$ lies in $\mathbb C^2$, hence
\[
\sum S_f = S_{v_1\ldots v_n} = \sum S'_f
\]
sum taken over all bounded faces of $H$. Now, for every $f$ we have $S_f \ge S'_f$ with strict inequality for at least one $f$, giving a contradiction.
\end{proof}

Theorem 1 follows immediately from Proposition 3.

To deduce Theorem 2 we observe that the area of $P$ is algebraic, so the areas of all faces are algebraic. Hence the variety $X$ is defined over $\overline{\mathbb Q}$ and we are done by the following well-known fact.

\begin{lemma}
If an algebraic variety defined over $\overline{\mathbb Q}$ has finitely many points over $\mathbb C$, then all these points are defined over $\overline{\mathbb Q}$.
\end{lemma}

\noindent{\bf Acknowledgements.}
The author is grateful to Dmitrii Korshunov, Misha Verbitsky, Altan Erdnigor and Ilya I. Bogdanov for interesting discussions.

\bibliographystyle{plain}
\bibliography{main}

{\small
\noindent {\sc Ivan Frolov\\
{\sc Instituto Nacional de Matem\'atica Pura e Aplicada (IMPA), Estrada Dona Castorina, 110, Jardim Bot\^anico, CEP 22460-320, Rio de Janeiro, RJ - Brasil\\
\tt  ivan.il.frolov@gmail.com }
}}

\end{document}